\documentclass[11pt]{amsart}
\usepackage{amsmath,amssymb,amsthm,enumitem,mathtools,tikz}
\usepackage[hmargin=20mm,vmargin=30mm]{geometry}
\usepackage[hidelinks]{hyperref}
\numberwithin{equation}{section}

\setlist[itemize]{leftmargin=*,labelindent=\parindent}
\setlist[enumerate]{leftmargin=*,labelindent=\parindent}

\theoremstyle{plain}
\newtheorem{thm}{Theorem}[section]
\newtheorem*{thm*}{Theorem}
\newtheorem{lem}[thm]{Lemma}
\newtheorem{prop}[thm]{Proposition}

\theoremstyle{remark}
\newtheorem*{rem}{Remark}

\newcommand\GL{\mathrm{GL}}
\newcommand\GSp{\mathrm{GSp}}
\newcommand\gl{\mathfrak{gl}}
\newcommand\M{\mathrm{M}}
\newcommand\SL{\mathrm{SL}}

\newcommand\C{\mathbb{C}}
\newcommand\F{\mathbb{F}}
\newcommand\Q{\mathbb{Q}}
\newcommand\Z{\mathbb{Z}}

\newcommand\cL{\mathcal{L}}
\newcommand\cM{\mathcal{M}}
\newcommand\cN{\mathcal{N}}
\newcommand\cR{\mathcal{R}}

\title{Representations of $\GL_2$ over $\Z/p^n\Z$ and supercongruences for hypergeometric polynomials}
\author{Atsushi Ichino}
\address{Department of Mathematics, Kyoto University, Kitashirakawa Oiwake-cho, Sakyo-ku, Kyoto 606-8502, Japan}
\email{ichino@math.kyoto-u.ac.jp}
\author{Kartik Prasanna}
\address{Department of Mathematics, University of Michigan, 2074 East Hall, 530 Church Street, Ann Arbor, MI 48109-1043, USA}
\email{kartikp@umich.edu}

\begin{document}

\begin{abstract}
For an odd prime $p$, we realize the trivial representation of $\GL_2(\Z/p^n\Z)$ on the free $\Z/p^n \Z$-module of rank one as a subquotient of a direct sum of symmetric power representations (twisted by appropriate powers of the determinant) of rank strictly greater than one.
The proof eventually reduces to establishing some novel supercongruences for hypergeometric polynomials.
\end{abstract}

\maketitle

\section{Introduction}

\subsection{Motivation and main result}

One of the most basic examples of a representation of a (non-abelian) group is a finite-dimensional (rational) representation of $\GL_2(\C)$ over $\C$.
Such a representation is always semisimple, and an irreducible representation (up to twist by a character) is classified by its dimension.
Namely, for any non-negative integer $k$, there exists a unique (up to twist by a character) irreducible representation of $\GL_2(\C)$ over $\C$ of dimension $k+1$.
Moreover, such an irreducible representation can be realized on the space $\C[x,y]_k$ of homogeneous polynomials of degree $k$.

For arithmetic applications, it is natural to consider finite-dimensional representations of $\GL_2(F)$ over an arbitrary field $F$.
When $F$ has characteristic zero, $\GL_2(F)$ has the same representation theory as $\GL_2(\C)$.
In particular, the representation $F[x,y]_k$ is always irreducible.
In contrast, when $F$ has positive characteristic, the representation theory of $\GL_2(F)$ is richer and more intricate.
For example, there are non-semisimple finite-dimensional representations of $\GL_2(F)$ over $F$, and the representation $F[x,y]_k$ itself is typically not irreducible or even semisimple.
For a finite field $F = \F_{p^n}$, where $p$ is a prime and $n \ge 1$ is an integer, these representations have been widely studied (see e.g.~\cite{ghate-jana}, \cite{ghate-vangala}, \cite{glover}, \cite{reduzzi}) and have played a significant role in arithmetic applications, such as the mod-$p$ and $p$-adic local Langlands correspondences for $\GL_2(\Q_p)$.

In this paper, we consider a more challenging setting: instead of a finite field $F = \F_{p^n}$, we take a finite ring $R=\Z/p^n \Z$, where $p$ is an odd prime, and study representations of $\GL_2(R)$ over $R$.
When $n>1$, the situation is further obscured and almost nothing seems to be known due to the following inherent difficulties:
\begin{itemize}
\item A finitely generated $R$-module is not necessarily free.
\item The map $R \rightarrow R, \, x \mapsto x^p$ is not a ring homomorphism.
\end{itemize}
We encountered this setting in our study of $p$-adic (resp.~mod-$p^n$) modular forms on $\GSp_4$, which take values in representations of $\GL_2$ over $\Z_p$ (resp.~$\Z/p^n \Z$). 
Specifically, we were interested in explicitly constructing an overconvergent solution to a certain $p$-adic differential equation on a $\GSp_4$-Shimura variety (arising from the Gauss-Manin connection) using successive mod-$p^n$ approximations.
This leads naturally to the study of differential (theta) operators between sheaves on this Shimura variety associated with the symmetric power 
representations $(\Z/p^n \Z )[x,y]_k$.  
Moreover, it becomes important to understand whether the sheaf associated with a given symmetric power representation (such as say the \emph{trivial} representation) can be realized as a subquotient of the sheaf associated with a different symmetric power representation. 
This question can be formulated in purely representation-theoretic terms: 
can the trivial representation of $\GL_2(\Z/p^n \Z)$ on the free $\Z/p^n \Z$-module of rank one be realized as a subquotient of the symmetric power representation $(\Z/p^n \Z )[x,y]_k$ (up to twist) for some $k \ge 1$?

For $n=1$, we restrict our attention to the case $k=2(p-1)$, as this is the one relevant to our main result.
In this case, the trivial representation of $\GL_2(\Z/p\Z)$ appears in
\[
 (\Z/p\Z)[x,y]_{2(p-1)}
\]
as a quotient, which is realized by the equivariant projection onto the line spanned by $x^{p-1} y^{p-1}$.
However, for $n>1$, it seems unlikely that any single representation $(\Z/p^n \Z)[x,y]_k$ contains the trivial representation as a subquotient (up to twist), which makes it difficult to find the correct formulation.
Nevertheless, we are able to precisely formulate and prove an analogous result for all $n > 1$.
To illustrate the nature of the main result, it suffices to focus on the case $n=2$.

\begin{thm*}
Assume that $p \ne 2$.
For $n=2$ (so that $R = \Z/p^2 \Z$), the trivial representation of $\GL_2(R)$ on the free $R$-module of rank one appears in
\[
 R[x,y]_{2(p-1)} \oplus R[x,y]_{2(p^2-p)} \oplus R[x,y]_{2(p^2-1)}
\]
as a subquotient, where each direct summand is twisted by an appropriate power of the determinant.
\end{thm*}

\begin{rem}
We emphasize that the above theorem does \emph{not} imply that the trivial representation appears as a subquotient in one of the three representations.
Although the Jordan-H\"older theorem shows that at least one of them contains a non-zero $R$-module as a subquotient on which $\GL_2(R)$ acts trivially, it does not guarantee that this module is free since $R$ is not a field.
In fact, it seems unlikely that the statement would still hold if any single summand were removed from the direct sum.
\end{rem}

In general (so that $R = \Z/p^n\Z$), to realize the trivial representation of $\GL_2(R)$, we need $2^n-1$ representations $R[x,y]_{2k_r}$ for $1 \le r \le 2^n-1$, where 
\[
 k_0 < k_1 < \cdots < k_{2^n-1}
\]
denote the first $2^n$ non-negative integers whose digits are either $0$ or $p-1$ in base $p$ (see Theorem \ref{t:main}).
The proof of this theorem eventually reduces to establishing supercongruences for certain hypergeometric polynomials (see Theorem \ref{t:hyp}).
Surprisingly, even the congruence for the seemingly simplest coefficients of these polynomials, which corresponds to the case $i=0$ in Theorem \ref{t:key}, turns out to be new (to the best of our knowledge) and rather difficult to prove: it states that 
\[
 \sum_{r=0}^{2^n-1} (-1)^r \binom{-1/2}{k_r/2} \equiv 0 \mod p^n.
\]

There is a large body of work on $p$-adic differential equations and related congruences involving hypergeometric functions, starting with the pioneering work of Dwork \cite{dwork}.
However, we have not yet found a connection between our result and previous work on this topic.
Discovering such a link would be very interesting.
It would also be interesting to generalize our result to other representations of $\GL_2$ or to other algebraic groups. 

In the following subsection, we describe how the above theorem is related to congruences for hypergeometric polynomials, and outline the main idea of the proof.

\subsection{Sketch of the proof}

To compare the three representations in the theorem stated in the previous subsection, it is better to realize them within a single representation, which is carried out in Section \ref{s:main}.
Inspired by the Fock model of the Weil representation, we are led to use an integral version of the Clebsch-Gordan decomposition
\begin{align*}
 \C[\M_2] & = \C[z_{11},z_{12}] \otimes_\C \C[z_{21},z_{22}] \\
 & = \bigoplus_{k_1=0}^\infty \bigoplus_{k_2=0}^\infty 
 \C[z_{11},z_{12}]_{k_1} \otimes_\C \C[z_{21},z_{22}]_{k_2} \\
 & \cong \bigoplus_{k_1=0}^\infty \bigoplus_{k_2=0}^\infty \bigoplus_{i=0}^{\min \{ k_1, k_2 \}} \C[x,y]_{k_1+k_2-2i}.
\end{align*}
Here we denote by $\C[\M_2]$ the space of polynomial functions of $z = \begin{psmallmatrix} z_{11} & z_{12} \\ z_{21} & z_{22} \end{psmallmatrix}$ and omit the necessary twists for each direct summand on the right-hand side.
For $R = \Z/p^n\Z$ and $k \ge 0$, we realize (the dual of) $R[x,y]_{2k}$ as $\cL_{2k}/p^n \cL_{2k}$ (up to twist), where $\cL_{2k} \subset \Z[\M_2]$ is a certain free submodule of rank $2k+1$ such that $\cL_{2k} \otimes_{\Z} \C \cong \C[x,y]_{2k}$ (up to twist).
For example, we take as $\cL_0$ the submodule generated by $\psi_0(z) = (z_{11} z_{22} - z_{12} z_{21})^{p^n}$.
Put $\cM_r = \cL_{2k_r} \otimes_\Z \Z[1/2]$ for $0 \le r \le 2^n-1$ and $\cM = \bigoplus_{r=1}^{2^n-1} \cM_r$.
Then we have natural equivariant injective maps
\[
 \cM_0/ (\cM_0 \cap p^n \cR) \rightarrow (\cM_0 \oplus \cM)/ ((\cM_0 \oplus \cM) \cap p^n \cR) \leftarrow \cM/ (\cM \cap p^n \cR),
\]
where $\cR = \Z[1/2][\M_2]$.
Since $\cM_0$ is generated by $\psi_0$, we have $\cM_0 \cap p^n \cR = p^n \cM_0$, so that the image $\cN_0$ of the first map is a free $R$-module of rank one.
If we can find a polynomial $\psi \in \cM$ such that the congruence 
\begin{equation}
\label{eq:psi}
 \psi_0 - \psi \in p^n \cR 
\end{equation}
holds, then the preimage $\cN$ of $\cN_0$ under the second map is also a free $R$-module of rank one, which implies that the preimage of $\cN$ in $\cM/p^n \cM$ gives rise to the desired subquotient.

The construction of $\psi$ is also given in Section \ref{s:main}, but it is extremely difficult to prove the above congruence.
In fact, by explicitly computing the coefficients of $\psi_0 - \psi$, we see that \eqref{eq:psi} is equivalent to 
\[
 \sum_{r=0}^{2^n-1} \sum_l (-1)^{r+k_r/2+l}
  2^{-k_r} \binom{k_r}{k_r/2+j}
  \binom{k_r-2j}{k_r-i+l} \binom{k_r+2j}{i-l} \binom{p^n-k_r}{l}
  \equiv 0 \mod p^n
\]
for all $i,j \ge 0$, where $l$ runs over all integers.
Although it is not impossible to prove this directly using a case-by-case argument for $n=2$, its extension to general $n$ seems hopeless.

Thus, a further reduction is necessary, which is carried out in Section \ref{s:some}.
By computing $\psi_0 - \psi$ in a different way and making a suitable change of variables, we reduce \eqref{eq:psi} to
\[
 \sum_{r=0}^{2^n-1} (-1)^r t^{(p^n-1-k_r)/2} {}_2F_1(-k_r/2,-(k_r-1)/2;1;1-t) \in p^n \Z[1/2][t]
\]
(see Theorem \ref{t:hyp}).
Here ${}_2F_1(a,b;c;t)$ denotes the hypergeometric function and each summand on the left-hand side is in fact a polynomial in $\Z[1/2][t]$.
Moreover, by computing the coefficients of the left-hand side, we further reduce this to
\[
 \sum_{r=0}^{2^n-1} (-1)^r \binom{k_r/2}{i} \binom{-1/2-i}{k_r /2}
 \equiv 0 \mod p^n
\]
for all $i \ge 0$ (see Theorem \ref{t:key}).
For $n=2$, its direct proof is much easier than that of the original congruence.
However, its extension to general $n$ still seems infeasible.

In Section \ref{s:more}, we prove Theorem \ref{t:key} as a special case of a further generalization (see Theorem \ref{t:gen1}), which takes the form
\[
 \sum_{r=0}^{2^n-1} (-1)^r \binom{m(r)}{i} \binom{m(2^a-1)-i}{m(r)}
 \equiv 0 \mod p^n
\]
for all $n \ge 1$, $a \ge n$, and $i \ge 0$, where $m$ is a certain monotonically increasing function.
The proof of Theorem \ref{t:gen1} is given by induction for each fixed $n,a,i$.
To make the inductive argument work, we introduce new variables $0 \le k \le n$ and $0 \le j \le n-k$, and formulate an intermediate statement in Theorem \ref{t:gen2}, which is too technical to state here.
Consequently, we proceed by double induction on $k$ and $j$ to prove Theorem \ref{t:gen1}, where the key ingredient is a result of Granville \cite{granville} on congruences for binomial coefficients modulo $p^n$.
We remark that $p=2$ is allowed at this stage, in which case Theorem \ref{t:gen1} implies that
\[
 \sum_{r=0}^{2^n-1} (-1)^r \binom{r}{i} \binom{2^a-1-i}{r} \equiv 0 \mod 2^n
\]
for all $n \ge 1$, $a \ge n$, and $i \ge 0$.
Surprisingly, we were not able to find a proof of this congruence in the literature.

\subsection*{Notation}

For any integers $n \ge 0$ and $k$, the binomial coefficient $\binom{n}{k}$ is given by
\[
\binom{n}{k} =
\begin{cases}
 \dfrac{n!}{k! (n-k)!} & \text{if $0 \le k \le n$;} \\
 0 & \text{if $k<0$ or $k>n$.}
\end{cases}
\]
If $k \ge 0$, then we may define
\[
 \binom{x}{k} = \frac{x (x-1) \cdots (x-k+1)}{k!}
\]
for an arbitrary number $x$.

Let $p$ be a prime.
We denote by $v_p$ the $p$-adic valuation.
For $x \in \Z_{(p)}$ and $n \ge 1$, we write $x \equiv 0 \mod p^n$ if $x \in p^n \Z_{(p)}$.
We define the $p$-deprived factorial by
\[
 n!_p = \prod_{\substack{1 \le i \le n \\ p \nmid i}} i
\]
for $n \ge 0$.

\section{Representations of $\GL_2$ over $\Z/p^n\Z$}
\label{s:main}

In this section, we state our main result and reduce it to congruences for certain polynomials.

\subsection{Symmetric and divided powers}

Let $\Delta$ be the monoid consisting of matrices $\alpha \in \M_2(\Z)$ such that $\det \alpha \ne 0$.
Define an action of $\Delta$ on $\Z[z_1,z_2]$ by 
\[
 (\alpha \phi)(z) = \phi(z\alpha)
\]
(regarded as functions of $z = (z_1,z_2)$) for $\alpha \in \Delta$ and $\phi \in \Z[z_1,z_2]$.
For example, the induced actions of $\begin{psmallmatrix} 0 & 0 \\ 1 & 0 \end{psmallmatrix}, \begin{psmallmatrix} 0 & 1 \\ 0 & 0 \end{psmallmatrix} \in \gl_2(\Z)$ are given by 
\[
 z_2 \frac{\partial}{\partial z_1}, \quad
 z_1 \frac{\partial}{\partial z_2},
\]
respectively.
For any integer $k \ge 0$, we denote by $L_k' \subset \Z[z_1,z_2]$ the module of homogeneous polynomials of degree $k$.
Let $L_k \subset L_k'$ be the submodule generated by $\{ \binom{k}{i} z_1^{k-i} z_2^i \, | \, 0 \le i \le k \}$.
Then $L_k$ and $L_k'$ are free $\Z$-modules of rank $k+1$ which are stable under the action of $\Delta$.
Moreover, if we define a non-degenerate bilinear pairing $(\cdot, \cdot ):L_k \times L_k' \rightarrow \Z$ by 
\[
 (\tbinom{k}{i} z_1^{k-i} z_2^i, z_1^{k-j} z_2^j) = 
\begin{cases}
 (-1)^i & \text{if $i+j=k$;} \\
 0 & \text{otherwise}
\end{cases}  
\]
for $0 \le i, j \le k$, then we have
\[
 (\alpha \phi, \alpha \phi') = (\det \alpha)^k (\phi, \phi')
\]
for $\alpha \in \Delta$, $\phi \in L_k$, and $\phi' \in L_k'$.

Fix a prime $p$ and an integer $n \ge 1$.
Put 
\[
 V_k = L_k/p^n L_k, \quad
 V'_k = L'_k/p^n L'_k,
\]
so that $V_k$ and $V_k'$ are free $\Z/p^n\Z$-modules of rank $k+1$, and can be regarded as representations of $\GL_2(\Z/p^n\Z)$.
For any integer $m$, we write $V_{k,m}$ for the representation of $\GL_2(\Z/p^n\Z)$ on $V_k$ twisted by $\det^m$.
Then the pairing $(\cdot,\cdot)$ induces an isomorphism
\[
 \operatorname{Hom}_{\Z/p^n\Z}(V_k',\Z/p^n\Z) \cong V_{k,-k}
\]
of representations of $\GL_2(\Z/p^n\Z)$.

Now we state our main result.
For any integer $r \ge 0$, we write $r = \sum_{i=0}^\infty r_i 2^i$ with $r_i \in \{ 0, 1 \}$ and put 
\begin{equation}
\label{eq:kr}
 k_r = (p-1) \sum_{i=0}^\infty r_i p^i. 
\end{equation}
Note that the non-negative integers $k_r$ are exactly those whose digits are either $0$ or $p-1$ in base $p$.

\begin{thm}
\label{t:main}
Assume that $p \ne 2$.
Then $V_{0,p^n}$ appears in 
\[
 \bigoplus_{r=1}^{2^n-1} V_{2k_r,p^n-k_r}
\]
as a subquotient.
\end{thm}

\subsection{Another realization}

To prove Theorem \ref{t:main}, we realize and compare the representations $V_{2k,m}$ in a single representation.
Define an action of $\Delta$ on $\Z[z_{11}, z_{12}, z_{21}, z_{22}]$ by 
\[
 (\alpha \phi)(z) = \phi(z\alpha)
\]
(regarded as functions of $z = \begin{psmallmatrix} z_{11} & z_{12} \\ z_{21} & z_{22} \end{psmallmatrix}$) for $\alpha \in \Delta$ and $\phi \in \Z[z_{11}, z_{12}, z_{21}, z_{22}]$.
For example, the induced actions of $\begin{psmallmatrix} 0 & 0 \\ 1 & 0 \end{psmallmatrix}, \begin{psmallmatrix} 0 & 1 \\ 0 & 0 \end{psmallmatrix} \in \gl_2(\Z)$ are given by 
\[
 D = z_{12} \frac{\partial}{\partial z_{11}} + z_{22} \frac{\partial}{\partial z_{21}}, \quad
 E = z_{11} \frac{\partial}{\partial z_{12}} + z_{21} \frac{\partial}{\partial z_{22}},
\]
respectively.
Define $\phi, \nu \in \Z[z_{11}, z_{12}, z_{21}, z_{22}]$ by
\[
 \phi(z) = z_{11}^2 + z_{21}^2, \quad
 \nu(z) = z_{11} z_{22} - z_{12} z_{21}.
\]
Note that $E \phi = D \nu = E \nu = 0$.
For any integers $k,l \ge 0$, put
\[
 \phi_{k,l} = \frac{1}{l!} D^l \phi^k.
\]
If $l \ge 2k+1$, then we have $\phi_{k,l} = 0$.

\begin{lem}
\label{l:phi_kl}
We have $\phi_{k,l} \in \Z[z_{11}, z_{12}, z_{21}, z_{22}]$.
\end{lem}

\begin{proof}
We have
\begin{align*}
 \phi_{k,l}(z) & = \frac{1}{l!} \sum_{i=0}^l \sum_{j=0}^k \binom{l}{i} \binom{k}{j} z_{12}^{l-i} z_{22}^i \frac{\partial^l}{\partial z_{11}^{l-i} \partial z_{21}^i} z_{11}^{2k-2j} z_{21}^{2j} \\
 & = \sum_{i=0}^l \sum_{j=0}^k \binom{2k-2j}{l-i} \binom{2j}{i} \binom{k}{j} z_{11}^{2k-2j-l+i} z_{12}^{l-i} z_{21}^{2j-i} z_{22}^i.
\end{align*}
This yields the assertion.
\end{proof}

For any integers $k,m \ge 0$, let $\cL_{2k,m} \subset \Z[z_{11}, z_{12}, z_{21}, z_{22}]$ be the submodule generated by $\{ \phi_{k,l} \cdot \nu^m \, | \, 0 \le l \le 2k \}$.
Then $\cL_{2k,m}$ is a free $\Z$-module of rank $2k+1$.

\begin{lem}
\label{l:another}
The module $\cL_{2k,m}$ is stable under the action of $\Delta$.
Moreover, we have 
\[
 \cL_{2k,m}/p^n \cL_{2k,m} \cong V_{2k,m}
\]
as representations of $\GL_2(\Z/p^n\Z)$.
\end{lem}

\begin{proof}
Let $\Delta' \subset \Delta$ be the submonoid consisting of diagonal matrices and put $\Gamma = \SL_2(\Z)$, so that $\Delta = \Gamma \Delta' \Gamma$.
Then it follows from the proof of Lemma \ref{l:phi_kl} that 
\[
 \alpha \phi_{k,l} = a_1^{2k-l} a_2^l \phi_{k,l}
\]
for $\alpha = \begin{psmallmatrix} a_1 & 0 \\ 0 & a_2 \end{psmallmatrix} \in \Delta'$ and $0 \le l \le 2k$.
Also, we have $\alpha \nu^m = (\det \alpha)^m \nu^m$ for $\alpha \in \Delta$.
This shows that $\mathcal{L}_{2k,m}$ is stable under the action of $\Delta'$.
Next we consider the action of $\Gamma$.
Let $H = ED-DE$ be the action of $\begin{psmallmatrix} 1 & 0 \\ 0 & -1 \end{psmallmatrix} \in \gl_2(\Z)$ on $\Z[z_{11}, z_{12}, z_{21}, z_{22}]$.
Then we have
\[
 H \phi_{k,l} = (2k-2l) \phi_{k,l}
\]
by the above formula.
Also, we have
\[
 D \phi_{k,l} = (l+1) \phi_{k,l+1}
\]
by definition, and 
\[
 E \phi_{k,l} = (2k-l+1) \phi_{k,l-1}
\]
by induction on $l$, where we interpret $\phi_{k,-1} = 0$.
From this, we deduce that 
\begin{align*}
 \begin{pmatrix} 1 & 0 \\ 1 & 1 \end{pmatrix} \phi_{k,l} 
 & = \sum_{i=0}^{\infty} \frac{1}{i!} D^i \phi_{k,l}
 = \sum_{i=0}^{2k-l} \binom{l+i}{i} \phi_{k,l+i}, \\
 \begin{pmatrix} 1 & 1 \\ 0 & 1 \end{pmatrix} \phi_{k,l} 
 & = \sum_{i=0}^{\infty} \frac{1}{i!} E^i \phi_{k,l}
 = \sum_{i=0}^l \binom{2k-l+i}{i} \phi_{k,l-i}.
\end{align*}
Since $\Gamma$ is generated by $\begin{psmallmatrix} 1 & 0 \\ 1 & 1 \end{psmallmatrix}, \begin{psmallmatrix} 1 & 1 \\ 0 & 1 \end{psmallmatrix}$, this shows that $\mathcal{L}_{2k,m}$ is stable under the action of $\Gamma$.
Hence the first assertion follows.

Let $f:\cL_{2k,m} \rightarrow L_{2k}$ be the isomorphism of $\Z$-modules given by 
\[
 f(\phi_{k,l} \cdot \nu^m) = \tbinom{2k}{l} z_1^{2k-l} z_2^l
\]
for $0 \le l \le 2k$.
Then the above formulas imply that
\[
 f(\alpha \phi) = (\det \alpha)^m \alpha f(\phi)
\]
for $\alpha \in \Delta$ and $\phi \in \cL_{2k,m}$.
Hence the second assertion follows.
\end{proof}

Now assume that $p \ne 2$ (so that the integers $k_r$ are even).
Put $\cR = \Z[1/2][z_{11}, z_{12}, z_{21}, z_{22}]$ and
\[
 \psi_r = (-1)^{k_r/2} 2^{-k_r} \phi_{k_r,k_r} \cdot \nu^{p^n-k_r} \in \cR
\]
for $0 \le r \le 2^n-1$.
Then we have
\begin{equation}
\label{eq:main}
 \sum_{r=0}^{2^n-1} (-1)^r \psi_r \in p^n \cR, 
\end{equation}
whose proof is postponed to Sections \ref{s:some} and \ref{s:more}.

In the rest of this section, we deduce Theorem \ref{eq:main} from \eqref{eq:main}.
For $0 \le r \le 2^n-1$, put
\[
 \cM_r = \cL_{2k_r,p^n-k_r} \otimes_\Z \Z[1/2],
\]
so that $\psi_r \in \cM_r$.
By Lemma \ref{l:another}, we have $\cM_r/p^n \cM_r \cong V_{2k_r,p^n-k_r}$ as representations of $\GL_2(\Z/p^n\Z)$.
Also, we have $p^n \cM_r \subset \cM_r \cap p^n \cR$, which is an equality if $r=0$ since $\cM_0$ is generated by $\psi_0 = \nu^{p^n}$.
Let $\cN_0$ be the image of the equivariant injective map
\[
 \cM_0/p^n \cM_0 = \cM_0/(\cM_0 \cap p^n \cR) \rightarrow (\cM_0 \oplus \cM) / ((\cM_0 \oplus \cM) \cap p^n \cR),
\]
where $\cM = \bigoplus_{r=1}^{2^n-1} \cM_r$.
Let $\cN$ be the preimage of $\cN_0$ under the equivariant injective map
\[
 \cM/(\cM \cap p^n \cR) \rightarrow (\cM_0 \oplus \cM) / ((\cM_0 \oplus \cM) \cap p^n \cR).
\]
Since $\cN_0$ is generated by the image of $\psi_0$, it follows from \eqref{eq:main} that the induced map $\cN \rightarrow \cN_0$ is an isomorphism.
Hence $\cN_0$ appears in the preimage of $\cN$ under the equivariant surjective map
\[
 \cM/p^n \cM \rightarrow \cM/(\cM \cap p^n \cR)
\]
as a quotient.
Since $\cN_0 \cong V_{0,p^n}$ and $\cM/p^n \cM \cong \bigoplus_{r=1}^{2^n-1} V_{2k_r,p^n-k_r}$, Theorem \ref{t:main} follows.

\section{Congruences for hypergeometric polynomials}
\label{s:some}

In this section, we reduce the proof of \eqref{eq:main} to congruences for certain hypergeometric polynomials.

\subsection{Reduction to congruences for hypergeometric polynomials}

Fix an odd prime $p$.
For any integer $r \ge 0$, let $k_r \ge 0$ be the even integer defined by \eqref{eq:kr}.
For any integer $m \ge 0$, put
\[
 F_m(t) = {}_2F_1(-m,-m+1/2;1;t), 
\]
where ${}_2F_1(a,b;c;t)$ denotes the hypergeometric function.
Note that $F_m(t)$ is in fact a polynomial given by
\[
 F_m(t) = \sum_{i=0}^m \binom{m}{i} \binom{m-1/2}{i} t^i \in \Z[1/2][t].
\]
Then we have the following congruences.

\begin{thm}
\label{t:hyp}
For any integer $n \ge 1$, we have
\[
 \sum_{r=0}^{2^n-1} (-1)^r t^{(p^n-1-k_r)/2} F_{k_r/2}(1-t) \in p^n \Z[1/2][t]. \]
\end{thm}

First we explain how this theorem implies \eqref{eq:main}.
Recall that 
\[
 \phi(z) = z_{11}^2 + z_{21}^2, \quad
 \nu(z) = z_{11} z_{22} - z_{12} z_{21}.
\]
Put
\[
 \psi(z) = z_{12}^2 + z_{22}^2, \quad
 \chi(z) = z_{11} z_{12} + z_{21} z_{22},
\]
so that 
\begin{equation}
\label{eq:Dphi}
 D \phi = 2 \chi, \quad
 D^2 \phi = 2 \psi, \quad
 D^3 \phi = 0.
\end{equation}
Put $\xi = \phi \psi = \nu^2 + \chi^2$ and $\eta = \chi^2$.

\begin{lem}
For any even integer $k \ge 0$, we have
\[
 \frac{1}{k!} D^k \phi^k = 2^k
 \sum_{i=0}^{k/2} \binom{k/2}{i} \binom{(k-1)/2}{i} \xi^i \eta^{k/2-i}.
\]
\end{lem}

\begin{proof}
By the Leibniz rule, we have
\[
 D^k \phi^k = \sum_{m_1, \dots, m_k} \frac{k!}{m_1! \cdots m_k!} D^{m_1} \phi \cdots D^{m_k} \phi, 
\]
where $m_1, \dots, m_k$ run over non-negative integers such that $m_1 + \dots + m_k = k$.
By \eqref{eq:Dphi}, we may assume that $m_1, \dots, m_k \le 2$.
For $m \in \{ 0, 1, 2 \}$, let $N_m$ be the number of $0 \le j \le k$ such that $m_j = m$.
Then we have $N_0 + N_1 + N_2 = k$ and $N_1 + 2N_2 = k$, so that $N_0 = N_2$ and $N_1 = k - 2N_2$.
Hence we have
\begin{align*}
 D^k \phi^k & = \sum_{i=0}^{k/2} \binom{k}{2i} \binom{2i}{i} \frac{k!}{2^i}
 \phi^i (2\chi)^{k-2i} (2\psi)^i \\
 & = 2^k k! \sum_{i=0}^{k/2} \binom{k}{2i} \binom{2i}{i}
2^{-2i} (\phi \psi)^i \chi^{k-2i}.
\end{align*}
Since
\[
 \binom{k}{2i} \binom{2i}{i} 2^{-2i}
 = \frac{k(k-1) \cdots (k-2i+1)}{2^{2i} (i!)^2}
 = \binom{k/2}{i} \binom{(k-1)/2}{i}, 
\]
the assertion follows.
\end{proof}

By this lemma, we have
\begin{align*}
 \psi_r & = (-1)^{k_r/2} \nu^{p^n-k_r} \sum_{i=0}^{k_r/2} \binom{k_r/2}{i} \binom{(k_r-1)/2}{i} \xi^i \eta^{k_r/2-i} \\
 & = (-1)^{(p^n-1)/2} \nu \eta^{(p^n-1)/2} (1-\xi \eta^{-1})^{(p^n-1-k_r)/2}
 \sum_{i=0}^{k_r/2} \binom{k_r/2}{i} \binom{(k_r-1)/2}{i} (\xi \eta^{-1})^i 
\end{align*}
in $\cR[1/\eta]$.
Thus, to prove \eqref{eq:main}, it suffices to show that 
\[
 \sum_{r=0}^{2^n-1} (-1)^r (1-t)^{(p^n-1-k_r)/2} F_{k_r/2}(t) \in p^n \Z[1/2][t].
\]
But this is equivalent, via the change of variable $t \mapsto 1-t$, to Theorem \ref{t:hyp}.

\subsection{Reduction to congruences for binomial coefficients}

To prove Theorem \ref{t:hyp}, we reduce it to congruences for binomial coefficients.

\begin{lem}
For $m \ge 0$, we have
\[
 F_m(1-t) = \sum_{i=0}^m (-1)^i \binom{m}{i} \binom{-1/2-i}{m} t^{m-i}.
\]
\end{lem}

\begin{proof}
By \cite[(3.17)]{gould}, we have
\[
 \sum_{i=0}^m \binom{m}{i} \binom{x}{i} t^i 
 = \sum_{i=0}^m \binom{m}{i} \binom{x+i}{m} (t-1)^{m-i}, 
\]
so that 
\[
 F_m(1-t) = \sum_{i=0}^m (-1)^{m-i} \binom{m}{i} \binom{m-1/2+i}{m} t^{m-i}.
\]
Since 
\[
 (-1)^m \binom{m-1/2+i}{m} = \binom{-i-1/2}{m}, 
\]
the assertion follows.
\end{proof}

By this lemma, Theorem \ref{t:hyp} is equivalent to 
\[
 \sum_{r=0}^{2^n-1} (-1)^r \sum_{i=0}^{(p^n-1)/2} (-1)^i \binom{k_r/2}{i} \binom{-i-1/2}{k_r/2} t^{(p^n-1)/2-i} \in p^n \Z[1/2][t]. 
\]
Hence Theorem \ref{t:hyp} follows from the following congruences.

\begin{thm}
\label{t:key}
For $n \ge 1$ and $i \ge 0$, we have
\[
 \sum_{r=0}^{2^n-1} (-1)^r \binom{k_r/2}{i} \binom{-1/2-i}{k_r/2} \equiv 0 \mod p^n.
\]
\end{thm}

Note that the left-hand side belongs to $\Z_{(p)}$ since $\binom{x}{k} \in \Z_p$ for $x \in \Z_p$ and $k \ge 0$.
The proof of this theorem will be given in the next section. 

\section{Congruences for binomial coefficients}
\label{s:more}

In this section, we give a proof of Theorem \ref{t:key}, which completes the proof of Theorem \ref{t:main}.
In fact, we prove the following generalization of Theorem \ref{t:key}.

\subsection{Generalized congruences}

Fix a prime $p$ and an integer $1 \le b \le p-1$.
For any integer $r \ge 0$, we write $r = \sum_{i=0}^\infty r_i 2^i$ with $r_i \in \{ 0, 1 \}$ and put 
\[
 m(r) = b \sum_{i=0}^\infty r_i p^i,
\]
which defines a monotonically increasing function.
If $p=2$ and $b=1$, then we have $m(r)=r$.
Also, for any integers $a \ge 0$ and $0 \le r \le 2^a-1$, we have
\begin{equation}
\label{eq:m}
 m(r) + m(d_a(r)) = m(2^a-1),
\end{equation}
where 
\[
 d_a(r) = 2^a-1-r.
\]
Note that the map $r \mapsto d_a(r)$ is an involution on $\{ 0,1,\dots,2^a-1\}$ induced by reflection about $2^{a-1}-1/2$.
In particular, we have
\begin{equation}
\label{eq:da}
 r < d_a(r) 
\end{equation}
for $a \ge 1$ and $0 \le r \le 2^{a-1} - 1$.

\begin{thm}
\label{t:gen1}
For $n \ge 1$, $a \ge n$, and $i \ge 0$, we have
\[
 \sum_{r=0}^{2^n-1} (-1)^r \binom{m(r)}{i} \binom{m(2^a-1)-i}{m(r)}
 \equiv 0 \mod p^n.
\]
\end{thm}

Note that Theorem \ref{t:key} is a special case of this theorem.
Indeed, if $p \ne 2$ and $b = (p-1)/2$, then we have $m(r) = k_r/2$ and $m(2^a-1) = (p^a-1)/2$.
By continuity of the function $x \mapsto \binom{x}{k}$ on $\Z_p$ for $k \ge 0$, we have
\[
 \binom{(p^a-1)/2 - i}{k_r/2} \equiv \binom{-1/2 - i}{k_r/2} \mod p^n
\]
for $a \gg 0$.
Hence Theorem \ref{t:gen1} implies Theorem \ref{t:key}.

Also, if $a=n$, then Theorem \ref{t:gen1} obviously holds.
Indeed, putting
\[
 X = \sum_{r=0}^{2^n-1} (-1)^r \binom{m(r)}{i} \binom{m(2^n-1)-i}{m(r)},
\]
we have
\[
 X = \sum_r (-1)^r \frac{(m(2^n-1) -i)!}{i! \cdot (m(r) - i)! \cdot (m(d_n(r)) -i)!}
\]
by \eqref{eq:m}, where $r$ runs over integers such that $0 \le r \le 2^n-1$ and $i \le \min \{ m(r), m(d_n(r)) \}$.
Changing the variable $r \mapsto d_n(r)$, we deduce that $X = -X$, so that $X=0$.
Hence we may assume that $a > n$.

\subsection{A result of Granville}

First we introduce some notation.
For any integers $n,j,l \ge 0$, put $t_j(n) = [n/p^j]$ and let $0 \le t_{j,l}(n) \le p^l-1$ be the integer such that $t_{j,l}(n) \equiv t_j(n) \mod p^l$.
Namely, if we write $n = \sum_{i=0}^\infty n_i p^i$ with $n_i \in \{0,1,\dots,p-1\}$, then we have
\[
 t_j(n) = \sum_{i=j}^\infty n_i p^{i-j}, \quad
 t_{j,l}(n) = \sum_{i=j}^{j+l-1} n_i p^{i-j}.
\]
Note that
\begin{equation}
\label{v_p(n!)}
 v_p(t_j(n)!) = \sum_{l=1}^\infty t_{j+l}(n)
\end{equation}
and 
\begin{equation}
\label{eq:t}
 t_j(n) + t_j(n') \le t_j(n+n') 
\end{equation}
for $n,n' \ge 0$.
Put $\epsilon_j(n,n') = 1$ if there is a carry in the $j$th digit when adding $n$ and $n'$ in base $p$, and $\epsilon_j(n,n') = 0$ otherwise.
Then we have
\[
 t_{j,1}(n+n') = t_{j,1}(n) + t_{j,1}(n') + \epsilon_{j-1}(n,n') - p \epsilon_j(n,n'),
\]
where we interpret $\epsilon_{-1}(n,n') = 0$.

\begin{lem}[{\cite[(2.4)]{granville}}]
\label{l:carry}
For $n,n',j \ge 0$, we have
\[
 t_j(n+n') - t_j(n) - t_j(n') = \epsilon_{j-1}(n,n').
\]
In other words, the equality in \eqref{eq:t} holds if and only if there is no carry in the $(j-1)$th digit when adding $n$ and $n'$ in base $p$.
\end{lem}

In \cite{granville}, Granville generalized Lucas' theorem on binomial coefficients modulo $p$ and gave a remarkable formula for binomial coefficients modulo an arbitrary power of $p$.
The proof of his formula is based on the following, which also plays a key role in the proof of Theorem \ref{t:gen1}.

\begin{prop}[{\cite[Corollary 1]{granville}}]
\label{p:granville}
For $n,j,l \ge 0$, we have
\[
 t_j(n)!/t_{j+1}(n)! \, p^{t_{j+1}(n)} = t_j(n)!_p \equiv
 \delta^{t_{j+l}(n)} t_{j,l}(n)!_p \mod p^l,
\]
where 
\[
 \delta = 
 \begin{cases}
  1 & \text{if $p=2$ and $l \ge 3$;} \\
  -1 & \text{otherwise.}
 \end{cases}
\]
\end{prop}

\subsection{Further generalizations}

We will obtain Theorem \ref{t:gen1} as a special case of more general congruences.
Fix $n \ge 1$ and $a > n$.
For $0 \le r \le 2^n-1$ (so that $r < d_a(r)$) and $i, j \ge 0$, put
\[
 x_j(r,i) = 
 \begin{cases}
  \dfrac{t_j(m(2^a-1) - i)!}{t_j(i)! \cdot t_j(m(r) - i)! \cdot t_j(m(d_a(r)) - i)!} & \text{if $i \le m(r)$;} \\
  0 & \text{if $i > m(r)$.}
 \end{cases}
\]
Note that $x_j(r,i) \in \Z_{(p)}$.
Indeed, if $i \le m(r)$, then putting 
\[
 e_j(r,i) = t_j(m(2^a-1) - i) - t_j(i) - t_j(m(r) - i) - t_j(m(d_a(r)) - i), 
\]
we have $e_j(r,i) \ge 0$ by \eqref{eq:m} and \eqref{eq:t}, and 
\begin{equation}
\label{v_p(x)}
 v_p(x_j(r,i)) = \sum_{l=1}^{\infty} e_{j+l}(r,i) \ge 0 
\end{equation}
by \eqref{v_p(n!)}.

For $0 \le k \le n$ and $0 \le r \le 2^{n-k}-1$, we define a subset $S_k(r)$ of $\{ 0, 1, \dots, 2^n-1 \}$ as follows:
\begin{itemize}
\item If $k=0$, then we put $S_0(r) = \{ r \}$.
\item If $1 \le k \le n$, then we put $S_k(r) = S'_k(r) \cup S''_k(r)$, where
\begin{align*}
 S'_k(r) & = \{ r+2^{n-k+1} r' \, | \, 0 \le r' \le 2^{k-1}-1\}, \\
 S''_k(r) & = \{ d_n(s) \, | \, s \in S'_k(r) \}.
\end{align*}
Since $d_n(s) = d_{n-k+1}(r) + 2^{n-k+1} d_{k-1}(r')$ for $s = r+2^{n-k+1} r' \in S_k(r)$, we can also write
\[
 S''_k(r) = \{ d_{n-k+1}(r) + 2^{n-k+1} r' \, | \, 0 \le r' \le 2^{k-1}-1 \}.
\]
\end{itemize}
Note that
\begin{equation}
\label{eq:S}
 S_k(r) \cup S_k(d_{n-k}(r)) = S_{k+1}(r)
\end{equation}
for $0 \le k < n$ and $0 \le r \le 2^{n-k-1}-1$.
Indeed, the assertion is obvious for $k=0$, whereas
\begin{align*}
 S'_k(r) \cup S''_k(d_{n-k}(r)) & = S'_{k+1}(r), \\
 S''_k(r) \cup S'_k(d_{n-k}(r)) & = S''_{k+1}(r)
\end{align*}
for $1 \le k < n$.

For $i,j \ge 0$, put
\[
 X_{j,k}(r,i) = \sum_{s \in S_k(r)} (-1)^s x_j(s,i).
\]
If $k = 0$ and $i > m(r)$, or $1 \le k \le n$ and $i>m(d_n(r))$, then we have
\begin{equation}
\label{eq:X}
 X_{j,k}(r,i) = 0
\end{equation}
by definition.
Note that the left-hand side of Theorem \ref{t:gen1} is equal to $X_{0,n}(0,i)$.
In particular, the following implies Theorem \ref{t:gen1}.

\begin{thm}
\label{t:gen2}
For $0 \le k \le n$, $0 \le r \le 2^{n-k}-1$, $i \ge 0$, and $0 \le j \le n-k$, we have
\[
 X_{j,k}(r,i) \equiv 0 \mod p^k.
\]
\end{thm}

\subsection{Proof of Theorem \ref{t:gen2}}

Fix $n \ge 1$ and $a > n$.
For $0 \le r \le 2^n-1$ and $i, j \ge 0$, put
\[
 y_j(r,i) = 
 \begin{cases}
  x_j(r,i)/x_{j+1}(r,i) & \text{if $i \le m(r)$;} \\
  0 & \text{if $i > m(r)$.}
 \end{cases}
\]
Note that $y_j(r,i) \in \Z_{(p)}$.
Indeed, if $i \le m(r)$, then we have $v_p(y_j(r,i)) = e_{j+1}(r,i) \ge 0$ by \eqref{v_p(x)}.
Moreover, if $i \le m(r)$, then it follows from Proposition \ref{p:granville} that
\begin{equation}
\label{eq:y}
 y_j(r,i) / p^{e_{j+1}(r,i)} \equiv \delta^{e_{j+l}(r,i)} 
 \frac{t_{j,l}(m(2^a-1)-i)!_p}{t_{j,l}(i)!_p \cdot t_{j,l}(m(r)-i)!_p \cdot t_{j,l}(m(d_a(r))-i)!_p} \mod p^l
\end{equation}
for $l \ge 0$.

\begin{lem}
\label{l:y}
For $0 \le k < n$, $0 \le r \le 2^{n-k-1}-1$, $0 \le i \le m(d_n(d_{n-k}(r)))$, and $0 \le j < n-k$, we have
\[
 y_j(d_n(d_{n-k}(r)),i) \equiv y_j(d_n(r),i) \mod p^{n-j-k}.
\]
\end{lem}

\begin{proof}
Put $r' = d_{n-k}(r)$, $s = d_n(r)$, and $s'=d_n(r')$, so that the assertion of the lemma is
\[
 y_j(s',i) \equiv y_j(s,i) \mod p^{n-j-k}.
\]
In view of \eqref{eq:da}, the following figure describes the positions of $r,r',s,s'$ when $k \ge 1$.
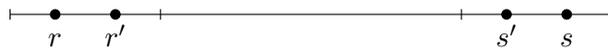
\vspace{5pt}
\[
\begin{tikzpicture}
 \draw (0,0) -- (8,0);
 \draw (0,2pt) -- (0,-2pt);
 \draw (2,2pt) -- (2,-2pt);
 \draw (6,2pt) -- (6,-2pt);
 \draw (8,2pt) -- (8,-2pt);
 \fill (0.6,0) circle[radius=2pt] node[below] {$r \vphantom{r'}$};
 \fill (1.4,0) circle[radius=2pt] node[below] {$r'$};
 \fill (6.6,0) circle[radius=2pt] node[below] {$s'$};
 \fill (7.4,0) circle[radius=2pt] node[below] {$s \vphantom{s'}$};
\end{tikzpicture}
\]
Since $s = 2^n-2^{n-k}+r'$ and $0 \le r' \le 2^{n-k}-1$, we have
\[
 m(s) = m(2^n-2^{n-k}) + m(r'). 
\]
Hence we have
\begin{align*}
 m(d_a(s')) & = m(2^a-2^n+r') = m(2^a-2^n) + m(r') \\
 & = m(2^a-2^n) - m(2^n-2^{n-k}) + m(s),
\end{align*}
so that 
\[
 m(2^n-2^{n-k}) + m(d_a(s')) - i = m(2^a-2^n) + m(s)-i.
\]
Similarly, we have
\[
 m(2^n-2^{n-k}) + m(d_a(s)) - i = m(2^a-2^n) + m(s')-i.
\]
This and Lemma \ref{l:carry} imply that for $0 \le l \le n-k$, we have
\begin{align*}
 t_l(m(2^n-2^{n-k})) + t_l(m(d_a(s')) - i) & = t_l(m(2^a-2^n)) + t_l(m(s)-i), \\
 t_l(m(2^n-2^{n-k})) + t_l(m(d_a(s)) - i) & = t_l(m(2^a-2^n)) + t_l(m(s')-i),
\end{align*}
so that 
\[
 e_l(s',i) = e_l(s,i).
\]
Moreover, we have
\begin{align*}
 t_{j,n-j-k}(m(d_a(s')) - i) & = t_{j,n-j-k}(m(s) - i), \\ 
 t_{j,n-j-k}(m(d_a(s)) - i) & = t_{j,n-j-k}(m(s') - i).
\end{align*}
From this and \eqref{eq:y}, the assertion follows.
\end{proof}

From now on, we fix $i \ge 0$.
For $j \ge 0$ and $0 \le k \le n$, we consider the following statement:
\begin{equation}
\text{$X_{j,k}(r,i) \equiv 0 \mod p^k$ for all $0 \le r \le 2^{n-k}-1$.} 
\tag*{$(*)_{j,k}$}
\end{equation}
We write $(*)_k$ for the collection of $(*)_{j,k}$ for all $0 \le j \le n-k$.

We need to show that $(*)_k$ holds for all $0 \le k \le n$.
We proceed by induction on $k$.
Obviously, $(*)_0$ holds.

\begin{lem}
\label{l:ind}
Fix $1 \le k \le n$.
Assume that $(*)_{k'}$ holds for all $0 \le k' < k$.
Then we have
\[
 X_{j,k}(r,i) \equiv y_j(d_n(r),i) X_{j+1,k}(r,i) \mod p^k
\]
for $0 \le j \le n-k$ and $0 \le r \le 2^{n-k}-1$.
\end{lem}

\begin{proof}
For $0 \le k' \le k$, we define a subset $R_{k'}(r)$ of $S_k(r)$ by
\[
 R_{k'}(r) = \{ s \in S_k(r) \, | \, 0 \le s \le 2^{n-k'} - 1 \}.
\]
In particular, we have $R_0(r) = S_k(r)$ and $R_k(r) = \{ r \}$.
Moreover, if $0 \le k' < k$, then we have
\begin{align*}
 R_{k'}(r) & = \{ r + 2^{n-k+1} r' \, | \, 0 \le r' \le 2^{k-k'-1}-1 \} \\
 & \cup \{ d_{n-k+1}(r) + 2^{n-k+1} r' \, | \, 0 \le r' \le 2^{k-k'-1}-1 \},
\end{align*}
so that
\begin{equation}
\label{eq:R}
 R_{k'}(r) = R_{k'+1}(r) \cup \{ d_{n-k'}(s) \, | \, s \in R_{k'+1}(r) \}. 
\end{equation}
Recall that 
\[
 X_{j,k}(r,i) = \sum_{s \in S_k(r)} (-1)^s y_j(s,i) x_{j+1}(s,i)
 = \sum_{s \in R_0(r)} y_j(s,i) X_{j+1,0}(s,i).
\]
It suffices to show that
\begin{equation}
 X_{j,k}(r,i) \equiv \sum_{s \in R_{k'}(r)} y_j(d_n(s),i) X_{j+1,k'}(s,i) \mod p^k
 \tag*{$(**)_{k'}$}
\end{equation}
for $1 \le k' \le k$.
Indeed, if $k'=k$, then the right-hand side is equal to $y_j(d_n(r),i) X_{j+1,k}(r,i)$, so that $(**)_k$ is simply a restatement of the lemma.

We proceed by induction on $k'$.
First consider the base case $k'=1$.
By \eqref{eq:R}, we have
\[
 X_{j,k}(r,i) = \sum_{s \in R_1(r)} (y_j(s,i) X_{j+1,0}(s,i) + y_j(d_n(s),i) X_{j+1,0}(d_n(s),i)).
\]
For $s \in R_1(r)$, if $i \le m(s)$, then noting that $k \le n-j$, we have
\[
 y_j(s,i) \equiv y_j(d_n(s),i) \mod p^k
\]
by Lemma \ref{l:y}.
This implies that 
\[
 y_j(s,i) X_{j+1,0}(s,i) \equiv y_j(d_n(s),i) X_{j+1,0}(s,i) \mod p^k.
\]
If $i>m(s)$, then the above congruence obviously holds since $X_{j+1,0}(s,i) = 0$ by \eqref{eq:X}.
Hence we have
\[
 X_{j,k}(r,i) \equiv \sum_{s \in R_1(r)} y_j(d_n(s),i) (X_{j+1,0}(s,i) + X_{j+1,0}(d_n(s),i)) \mod p^k.
\]
On the other hand, we have
\[
 X_{j+1,0}(s,i) + X_{j+1,0}(d_n(s),i) = X_{j+1,1}(s,i)
\]
by \eqref{eq:S}.
From this, we deduce that $(**)_1$ holds.

Now fix $1 \le k' < k$ and assume that $(**)_{k'}$ holds.
By \eqref{eq:R}, we have
\[
 X_{j,k}(r,i) \equiv \sum_{s \in R_{k'+1}(r)} (y_j(d_n(s),i) X_{j+1,k'}(s,i) + y_j(d_n(d_{n-k'}(s)),i) X_{j+1,k'}(d_{n-k'}(s),i)) \mod p^k.
\]
For $s \in R_{k'+1}(r)$, if $i \le m(d_n(d_{n-k'}(s)))$, then noting that $k-k' \le n-j-k'$, we have
\[
 y_j(d_n(d_{n-k'}(s)),i) \equiv y_j(d_n(s),i) \mod p^{k-k'}
\]
by Lemma \ref{l:y}.
Since $(*)_{k'}$ holds by assumption, this implies that 
\[
 y_j(d_n(d_{n-k'}(s)),i) X_{j+1,k'}(d_{n-k'}(s),i) \equiv 
 y_j(d_n(s),i) X_{j+1,k'}(d_{n-k'}(s),i) \mod p^k.
\]
If $i > m(d_n(d_{n-k'}(s)))$, then the above congruence obviously holds since $X_{j+1,k'}(d_{n-k'}(s),i) = 0$ by \eqref{eq:X}.
Hence we have
\[
 X_{j,k}(r,i) \equiv \sum_{s \in R_{k'+1}(r)} y_j(d_n(s),i) (X_{j+1,k'}(s,i) + X_{j+1,k'}(d_{n-k'}(s),i)) \mod p^k.
\]
On the other hand, we have
\[
 X_{j+1,k'}(s,i) + X_{j+1,k'}(d_{n-k'}(s),i) = X_{j+1,k'+1}(s,i)
\]
by \eqref{eq:S}.
From this, we deduce that $(**)_{k'+1}$ holds.

This completes the proof. 
\end{proof}

\begin{lem}
\label{l:base}
Fix $1 \le k \le n$.
Assume that $(*)_{k'}$ holds for all $0 \le k' < k$.
Then we have
\[
 X_{n-k,k}(r,i) \equiv 0 \mod p^k
\]
for $0 \le r \le 2^{n-k}-1$.
\end{lem}

\begin{proof}
By Lemma \ref{l:ind}, we have
\[
 X_{n-k,k}(r,i) \equiv y_{n-k}(d_n(r),i) X_{n-k+1,k}(r,i) \mod p^k.
\]
Since
\[
 X_{n-k+1,k}(r,i) = X_{n-k+1,k-1}(r,i) + X_{n-k+1,k-1}(d_{n-k+1}(r),i)
\]
by \eqref{eq:S}, and $(*)_{k-1}$ holds by assumption, we have
\[
 X_{n-k+1,k}(r,i) \equiv 0 \mod p^{k-1}.
\]
If $y_{n-k}(d_n(r),i) \equiv 0 \mod p$, then this implies that $X_{n-k,k}(r,i) \equiv 0 \mod p^k$.

Now assume that $y_{n-k}(d_n(r),i) \not\equiv 0 \mod p$ (so that $i \le m(d_n(r))$).
Then for any $s \in S_k(r)$ such that $i \le m(s)$, we show that 
\begin{equation}
\label{eq:y-modp}
 y_{n-k}(s,i) \not\equiv 0 \mod p
\end{equation}
by backward induction on $s$.
By assumption, \eqref{eq:y-modp} holds for the base case $s = d_n(r)$ (which is the largest integer in $S_k(r)$).
Fix $s_0 \in S_k(r)$ such that $s_0 < d_n(r)$ and $i \le m(s_0)$, and assume that \eqref{eq:y-modp} holds for all $s \in S_k(r)$ such that $s_0 < s \le d_n(r)$.
Let $k_0$ be the largest integer such that $d_n(s_0) \le 2^{n-k_0}-1$ (so that $d_n(s_0) \in R_{k_0}(r)$ with the notation of the proof of Lemma \ref{l:ind}).
Since $s_0 < d_n(r)$, we have $k_0 < k$.
By \eqref{eq:R}, we may write $d_n(s_0) = d_{n-k_0}(r_0)$ for some $r_0 \in R_{k_0+1}(r)$ (so that $r_0 \le 2^{n-k_0-1} -1$ and $s_0 = d_n(d_{n-k_0}(r_0)) < d_n(r_0)$).
Then we have
\[
 y_{n-k}(s_0,i) \equiv y_{n-k}(d_n(r_0),i) \mod p
\]
by Lemma \ref{l:y}.
Hence \eqref{eq:y-modp} holds for $s=s_0$ by the induction hypothesis.

By \eqref{eq:y-modp}, we have
\[
 e_{n-k+1}(s,i) = v_p(y_{n-k}(s,i)) = 0
\]
for $s \in S_k(r)$ such that $i \le m(s)$.
Since $e_{n-k+1}(s,i) = e' + e''$ with
\begin{align*}
 e' & = t_{n-k+1}(m(s)) - t_{n-k+1}(i) - t_{n-k+1}(m(s) -i) \ge 0, \\
 e'' & = t_{n-k+1}(m(2^a-1)-i) - t_{n-k+1}(m(d_a(s))-i) - t_{n-k+1}(m(s)) \ge 0
\end{align*}
by \eqref{eq:m} and \eqref{eq:t}, we have
\begin{align}
\label{eq:e'}
 t_{n-k+1}(m(s)) & = t_{n-k+1}(i) + t_{n-k+1}(m(s) -i), \\
\label{eq:e''}
 t_{n-k+1}(m(2^a-1)-i) & = t_{n-k+1}(m(d_a(s))-i) + t_{n-k+1}(m(s))
\end{align}
for $s \in S_k(r)$ such that $i \le m(s)$.

We will show that the assumption $y_{n-k}(d_n(r),i) \not\equiv 0 \mod p$ forces $X_{n-k+1,k}(r,i) = 0$.
Recall that 
\begin{align*}
 X_{n-k+1,k}(r,i) & = \sum_{s \in S'_k(r)} (-1)^s x_{n-k+1}(s,i) 
 + \sum_{s \in S''_k(r)} (-1)^s x_{n-k+1}(s,i) \\
 & = \sum_{s \in S'_k(r)} (-1)^s x_{n-k+1}(s,i) 
 - \sum_{s \in S'_k(r)} (-1)^s x_{n-k+1}(d_n(s),i).
\end{align*}
For $s = r + 2^{n-k+1} r' \in S'_k(r)$ with $0 \le r' \le 2^{k-1}-1$, 
put $s' = r + 2^{n-k+1} d_{k-1}(r') \in S'_k(r)$.
Note that $d_n(s') = d_{n-k+1}(r) + 2^{n-k+1} r'$, so that $d_n(s') > s$ and
\begin{equation}
\label{eq:t-consecutive}
 t_{n-k+1}(m(d_n(s'))) = t_{n-k+1}(m(s)).
\end{equation}
Since the map $s \mapsto s'$ is an involution on $S'_k(r)$, it suffices to show that 
\begin{equation}
\label{eq:x}
 x_{n-k+1}(d_n(s'),i) = x_{n-k+1}(s,i).
\end{equation}

First assume that $i \le m(s)$.
By \eqref{eq:e'}, \eqref{eq:e''}, and \eqref{eq:t-consecutive}, we have
\begin{align*}
 t_{n-k+1}(m(d_n(s')) -i) & = t_{n-k+1}(m(s) -i), \\
 t_{n-k+1}(m(d_a(d_n(s')))-i) & = t_{n-k+1}(m(d_a(s))-i).
\end{align*}
Hence \eqref{eq:x} follows from the definition.

Next assume that $m(s) < i \le m(d_n(s'))$.
Write $i = u + p^{n-k+1} u'$ with $0 \le u \le p^{n-k+1}-1$ and $u' \ge 0$.
Since
\begin{align*}
 m(s) & = m(r) + p^{n-k+1} m(r'), \\
 m(d_n(s')) & = m(d_{n-k+1}(r)) + p^{n-k+1} m(r'),
\end{align*}
we have
\[
 m(r) < u \le m(d_{n-k+1}(r)), \quad
 u' = m(r'), 
\]
and 
\[
 t_{n-k+1}(m(d_n(s'))) = m(r').
\]
Also, we have
\begin{align*}
 m(2^a-1) - i & = m(2^{n-k+1}-1) + m(2^a - 2^{n-k+1}) - i \\
 & = m(2^{n-k+1}-1) - u + p^{n-k+1}(m(2^{a-n+k-1}-1) - u'), 
\end{align*}
so that
\[
 t_{n-k+1}(m(2^a-1)-i) = m(2^{a-n+k-1}-1) - u'.
\]
Moreover, we have $d_a(d_n(s')) = r + 2^{n-k+1}(d_{a-n+k-1}(r'))$ and 
\[
 m(d_a(d_n(s'))) - i = m(r) - u + p^{n-k+1}(m(d_{a-n+k-1}(r')) - u'),
\]
so that
\[
 t_{n-k+1}(m(d_a(d_n(s'))) - i) < m(d_{a-n+k-1}(r')) - u'
\]
since $m(r)-u<0$.
Hence we have
\begin{align*}
 t_{n-k+1}(m(d_n(s'))) + t_{n-k+1}(m(d_a(d_n(s'))) - i)
 & < m(r') + m(d_{a-n+k-1}(r')) - u' \\
 & = m(2^{a-n+k-1} - 1) - u' \\
 & = t_{n-k+1}(m(2^a-1)-i).
\end{align*}
This contradicts \eqref{eq:e''} (with $s$ replaced by $d_n(s')$), which implies that the case $m(s) < i \le m(d_n(s'))$ cannot occur under the assumption $y_{n-k}(d_n(r),i) \not\equiv 0 \mod p$.

Finally, if $i>m(d_n(s'))$, then both sides of \eqref{eq:x} are equal to $0$.
This completes the proof.
\end{proof}

Fix $1 \le k \le n$.
Assume that $(*)_{k'}$ holds for all $0 \le k' < k$.
Then $(*)_{n-k,k}$ holds by Lemma \ref{l:base}.
Hence we deduce from Lemma \ref{l:ind} and backward induction on $j$ that $(*)_{j,k}$ holds for all $0 \le j \le n-k$.
Namely, $(*)_k$ holds.
This completes the proof of Theorem \ref{t:gen2}.

This also completes the proof of Theorem \ref{t:main}.

\subsection*{Acknowledgments}

We thank Stefano Morra for pointing us to several of the references below on modular representations and for useful comments on a previous version of the article. 
A.I. was partially supported by JSPS KAKENHI Grant Number 19H01781.
K.P. was partially supported by NSF grant DMS 2001293.

\end{document}